\documentclass[11pt]{article}
\usepackage{my_preamble}
\usepackage{amssymb, amsmath, amsfonts, stmaryrd}
\usepackage{amscd}
\usepackage{amscd}

\newcommand{\G}{\mathcal{G}}

\renewcommand{\rho}{\varrho}

\renewcommand{\heq}{\mathrm{heq}}
\newcommand{\WCb}{\mathrm{WCb}}

\begin{document}

\title{On Weak Elimination of Hyperimaginaries and its Consequences}
\author{Donald Brower\footnote{University of Notre Dame}\\
Cameron Donnay Hill\footnote{Correspondence to:   University of Notre Dame,
Department of Mathematics,
255 Hurley, Notre Dame, IN 46556. Email: \texttt{cameron.hill.136@nd.edu}}\\
}

\maketitle
\begin{abstract}
We analyze the notion of \emph{weak elimination of hyperimaginaries} (WEHI) in simple theories. A key observation in the analysis is a characterization of WEHI in terms of forking dependence -- a condition we dub \emph{dependence-witnessed-by-imaginaries} (DWIP). 
Generalizing results of \cite{adler-thesis} and \cite{ealy-thesis}, we show that in a simple theory with WEHI, forking and \th-forking coincide. We also show that, conversely, the equivalence of $\indep$ and $\thind$ is (almost) sufficient for WEHI. Thus, the WEHI and the statement $\indep=\thind$ are morally equivalent. 
As a further application of our technology, we demonstrate stable forking for 1-based theories of finite SU-rank that have WEHI. 
\end{abstract}

\section*{Introduction}

Canonical bases were and are a key notion in geometric stability theory. Consequently, recovering objects for simple, possibly unstable, theories with the same or similar properties was necessary in order to carry over existing arguments from the stable case. This led to the development of \emph{hyperimaginaries} for simple theories. Unfortunately, hyperimaginaries have few of the desirable properties of imaginaries relative to a given first-order theory $T$, and examples of simple theories that \emph{do not} eliminate hyperimaginaries are few on the ground. There are also several results in the literature which take elimination of hyperimaginaries as a hypothesis.

As it is not known whether all simple theories eliminate hyperimaginaries, we hope that investigating fragments of elimination of hyperimaginaries will lead to a clearer picture of the whole. In this article, we investigate a fragment called \emph{weak elimination of hyperimaginaries}. Our first result along this line is a characterization of WEHI itself in terms of forking dependence. This characterization suggests a number of ways in which WEHI plays a role in questions more specifically pertaining to forking-dependence in simple theories. Indeed, we are able to show not only that $\indep=\thind$ in any simple theory with WEHI, but that if $\indep=\thind$ and $\thind$ is canonical, then the theory in question must have WEHI.

\subsection*{Stable-forking in simple theories}

Recall that the class of simple theories is a proper extension of the class of stable theories. While the former can be characterized by, among other things, the symmetry of non-forking independence, forking in a stable theory (deviation from the definition of the type) is often rather easier to identify and smoother to work with than forking in a simple theory. For this reason, one might hope that forking in a simple unstable theory is always at least \emph{witnessed} by a stable formula. Also, although there is a plethora of examples of simple unstable theories, many of these seem to be essentially stable up to some kind of ``noise.'' One way to formalize this intuition is found in \cite{AroundStableForking}:
\begin{defn}
A simple theory $T$ has \emph{stable forking} if whenever $q(x)$ is a complete type over a model $M$, $A \subset M$ and $q$ forks over $A$ then there is is a stable formula
$\psi(x;b) \in q$ which forks over $A$.
\end{defn}
The Stable Forking Conjecture is simply that every simple theory has stable forking.
The conjecture is unresolved---and not universally believed true---but it has been verified in several special classes including 1-based supersimple theories and stable theories expanded with a generic predicate.

Although the result is fairly well-known to model-theorists, demonstrations of the fact that 1-based supersimple theories have stable forking have all involved highly technical machinery spread over several publications. The original proof runs through \cite{KimInThere}, where Kim showed that  1-based simple theories with elimination of hyperimaginaries (EHI) have stable forking, and it terminates in a very technical paper, \cite{BPW}, where it is shown that all supersimple theories have EHI. In the standard textbook on simple theories, \cite{Wagner-SimpleTheories}, it is proved as a sort of afterthought in an extended treatment of elimination of hyperimaginaries. In this note, we present some streamlined proofs of stable forking in 1-based theories of finite $SU$-rank. While still building on \cite{BPW}, we avoid a great deal of technical machinery.

The demonstrations in this article arose partially out of an attempt to obtain the results of \cite{Peretz} on stable forking for elements of ``low $SU$-rank'' in $\omega$-categorical simple theories while dropping the $\omega$-categoricity assumption.
The original project was successful in that we proved (section \ref{sec:elements_of_rank_2}) that in any simple theory with WEHI, if a rank-2 element $a$ forks with $b$, and $b$ has finite SU-rank, then $tp(a/b)$ contains a stable forking formula. We extend the method that worked for rank-2 elements against finite-rank elements to work for all finite-ranked elements in 1-based theories with WEHI (section \ref{sec:finite-su-rank}).

\bigskip

In what follows we assume a familiarity with common conventions of model theory .
We work in a suitably saturated model $\MM$;
we will use the terms forking, simple, supersimple, $SU$-rank,  canonical base, and so on freely.
One reference  for this background material is \cite{Wagner-SimpleTheories}. All that said, we include a very brief reminder on the basics of hyperimaginaries because the word ``hyperimaginary'' makes at least one of the authors a bit queasy.

\section{Weak elimination of hyperimaginaries}

In this short section, we present (in the first subsection) a review of hyperimaginaries in simple theories leading to the definition of weak elimination of hyperimaginaries. In the second subsection, we show that WEHI can be rephrased in a fairly inuitive way in terms of forking dependence.

\subsection{Background on hyperimaginaries}

This subsection presents background on hyperimaginaries. There are no new results contained here, and readers familiar with hyperimaginaries may wish to skip to the next subsection.
For further detail, we refer the reader to \cite{HKP}, \cite{BPW}, and \cite{Wagner-SimpleTheories}.

Stability theory makes much use of imaginary elements, each of which is a name for a class of a definable equivalence relation.
Imaginary elements can be ``added'' to a theory by adding sorts to represent each definable equivalence relation, and this
 ``eq'' construction preserves all of the relevant properties of the original first-order theory.

We can, then, do the same thing for type-definable equivalence relations,
calling the name of such an equivalence class a \emph{hyperimaginary}.
If $E(\bar x, \bar y)$ is a type definable equivalence relation
with $\bar x = (x_i : i \in I)$ and $\bar y = (y_i : i \in I)$ for some (possibly infinite) linearly ordered index set $I$, and $\bar a = (a_i : i \in I)$ is some tuple of elements from $M^\eq$ then $\bar a_E$ is the hyperimaginary element representing the class of all tuples $E$-related to $\bar a$.
It can be shown that any type-definable equivalence relation is equivalent to an infinite conjunction of equivalence relations defined by a countable types, so  it suffices to add names only for the equivalence relations defined by countably-infinite conjunctions of formulas.

A model $M$ enriched with these names is denoted $M^\heq$.
It's not hard to see that  $M \subset M^\eq \subset M^\heq$ up to some natural and obvious identifications.
Automorphisms of $M$ lift to $M^\heq$, so if $e = \bar a_E$ is a hyperimaginary element
we say an automorphism fixes $e$ iff it fixes the $E$ class of $\bar a$ as a set.
This is an annoyance of dealing with hyperimaginaries: the formulas comprising the equivalence relation $E$ do not need to be $e$-invariant.
The lifting of automorphisms does let us extend definable closure and algebraic closure to $M^\heq$; we will use $\dcl$ and $\acl$ to refer to these extensions for this paper.

Unfortunately, we cannot think of $M^\heq$ as a first-order structure. 
The main issue is that type-definable sets are closed but not clopen, so the negation
of such is not, in general,  type-definable.
However, a fragment of logic can be developed.
Specifically, we can extend the notions of \emph{type} and \emph{forking} to hyperimaginaries.
If $a$ and $b$ are hyperimaginaries then there are equivalence relations $E(x,y)$ and $F(w,z)$ and tuples $\bar a = (a_i : i \in I)$, $\bar b = (b_i : i \in J)$ such that $a = \bar a_E$, $b = \bar b_F$.
The type $\tp(a/b)$ can then be expressed 
as the union of the partial types over $\bar b$,
\[
\exists wz [E(x,w) \land F(z,\bar b) \land \varphi(w',z')]
\]
for all $\varphi$ such that the partial type is satisfied by $a$.
In the equation above $w'$ and $z'$ represent some finite subtuples of $w$ and $z$.
We could have chosen to base the type on any tuple $\bar b^*$ so long as $\bar b^*$ and $\bar b$ are $F$-equivalent.
This definition is the right one since it preserves the idea of types representing orbits of elements over a base set.
\begin{fact}\cite[Proposition 1.4]{HKP}
    $a'$ realizes $\tp(a/b)$ if and only if there is an automorphism $f$ 
fixing $b$ such that $f(a) = a'$.
\end{fact}
From the notion of type one can derive the notion of an indiscernible sequence of hyperimaginaries.
We then say a type $p(x,b) = \tp(a/b)$ divides over a set $C$ if there is a sequence
$(b_i : i < \omega)$ in $\tp(b/C)$ indiscernible over $C$ such that
$\bigcup_i p(x,b_i)$ is inconsistent.
A type forks if it implies a finite disjunction of types which divide.

Just as it is terribly convenient to work in a theory which eliminates imaginaries, it is very nice to work with theories that eliminate hyperimaginaries.
We say that a theory $T$ eliminates hyperimaginaries (has EHI) if for every hyperimaginary $e\in\MM^\heq$, there is a set of imaginaries $\bar b\subset \MM^\eq$ such that $\dcl(\bar b)=\dcl(e)$.
In this article, we investigate an (apparently) weaker condition:
\begin{defn}
    We say a theory has \textit{weak elimination of hyperimaginaries (WEHI)} if
    for every hyperimaginary $e\in\M^\heq$, there is a set of imaginaries
    $\left.B\subset\M^\eq\right.$ such that $\bdd(B)=\bdd(e)$.
\end{defn}
If a given hyperimaginary $e$ has a set of imaginaries $B$ where $\bdd(B) = \bdd(e)$, then we say $e$ is weakly eliminated.
This notion is parallel to that of elimination of hyperimaginaries.

\subsection{Expressing WEHI in terms of forking}
\label{sec:dwip-and-wehi}

\begin{defn}
    We say a simple theory has the \textit{Dependence Witnessed by Imaginaries Property (DWIP)} if
    whenever $a \nindep_{C} b$ for hyperimaginary elements $a$ and $b$ and set $C$ then there is an imaginary element $d \in \acl(bC)$ such that $a \nindep_{C} d$.
\end{defn}

\begin{prop}\label{prop:DWIP--WEHI}
    Suppose $T$ is simple and has DWIP.
    Then for every hyperimaginary $e$ there is a set of imaginary elements $D$ such that $D \subset \dcl(e)$ and $e \in \bdd(D)$. In particular, $T$ has WEHI.
\end{prop}
\begin{proof}
    Fix an arbitrary hyperimaginary $e$.
    Let $C = \acl(e) \cap \M^\eq$.
    For each element $a \in C$,
    let $a'$ be the imaginary element naming all finitely many conjugates of $a$ over $C$.
    Let $D$ be the set of all tuples $a'$ for $a \in C$.
    By construction $D \subset \dcl(e)$.
    \begin{claim}
    $C = \acl^\eq(D)$.
    \end{claim}
    \begin{proof}[Proof of claim]
        It is clear $D \subset C$ so $\acl^\eq(D) \subset \acl^\eq(C) = C$.
        For the other direction, suppose $c \in C$.
        Then, by construction, there is a corresponding element $c' \in D$,
        and there is an algebraic formula which says $c$ is in the set named by $c'$.
        Hence $c \in \acl^\eq(D)$.
    \end{proof}
    \begin{claim}
    $A \indep_D e$ for any set $A$.
    \end{claim}
    \begin{proof}[Proof of claim]
        Suppose not.
        Then there is some set $A$ such that $A \nindep_{D} e$.
        By DWIP there is an imaginary element $d \in \acl(eD) \cap \M^\eq$
        such that $A \nindep_{D} d$.
        Because $D \subset \dcl(e)$, in fact $\acl(eD) = \acl(e)$.
        Since $d$ is imaginary $d \in C$, and so $\left.d \in \acl(D)\right.$.
        Hence $A \indep_{D} d$, contradicting our choice of $d$.
    \end{proof}
    Now consider the case $A = e$.
    Then $e \indep_D e$, which implies $e \in \bdd(D)$.
\end{proof}
\begin{prop}\label{prop:WEHI--DWIP}
    Weak elimination of hyperimaginaries implies DWIP.
\end{prop}
\begin{proof}
    Suppose $T$ weakly eliminates hyperimaginaries, and $a$,$b$ are hyperimaginaries such that 
    $a \nindep_{C} b$.
    We need to find an imaginary element $d \in \acl(bC)$ such that $a \nindep_{C} d$.
    WEHI gives a set of imaginary elements $B$ such that $\bdd(B) = \bdd(b)$.
    Hence, $a \nindep_C \bdd(B)$, implying $a \nindep_{C} B$.
    Then by the finite character of forking, there is a finite tuple $d \in B$ such that $a \nindep_{C} d$.
    $B$ is a set of imaginary elements, hence $B \subset \bdd(b)$ implies $B \subset \acl(b)$.
    Hence $d \in B \subseteq \acl(b)$ and we are done.
\end{proof}
\begin{cor}\label{cor:DWIP-is-WEHI}
    A simple theory $T$ has DWIP if any only if it has WEHI.
\end{cor}
\begin{proof}
Combine Propositions~\ref{prop:DWIP--WEHI} and~\ref{prop:WEHI--DWIP}
\end{proof}

\section{\th-Independence in simple theories}

In this section, we will show that WEHI is morally equivalent to the statement $\indep=\thind$. The missing ingredient is the \emph{intersection property} for an abstract independence relation. (An independence relation with the intersection property is called canonical.)
The result below summarizes the results in this section.
\begin{thm}
    Let $T$ be a simple theory.
    The following are equivalent.
    \begin{enumerate}
        \item $T$ has weak elimination of hyperimaginaries
        \item $T$ has the dependence-witnessed-by-imaginaries property.
        \item $\indep = \thind$ and $\thind$ is canonical in $T$.
    \end{enumerate}
\end{thm}
In general outline, the proof that 1,2$\implies$3 follows Adler's proof \cite{adler-thesis}that \emph{full} elimination of hyperimaginaries is sufficient for $\indep=\thind$ in a simple theory. The proof of the converse also uses some definitions from \cite{adler-thesis}, but the arguments seem to new.

\subsection{Forking and \th-Forking given WEHI}

Following both \cite{adler-thesis} and \cite{ealy-thesis}, we work in the context of abstract ternary relations on subsets of a model.
\begin{defn}
    Let $\oind$ be a ternary relation on subsets of a model. $\oind$ is an \textit{independence relation}
    if the following properties hold.
    For the following, $A,B,C,D$ are arbitrary sets
    \begin{itemize}
        \item Invariance: $A \oind_C B$ and $A'B'C' \equiv ABC \implies A' \oind_{C'} B'$
        \item Monotonicity: $A \oind_C B$, $A' \subset A$, $B' \subset B \implies A' \oind_C B'$
        \item Base Monotonicity: If $D \subset C \subset B$ and $A \oind_D B$ then $A \oind_C B$
        \item Transitivity: If $D \subset C \subset B$, $B \oind_C A$ and $C \oind_D A$ then $B \oind_D A$
        \item Normality: $A \oind_C B \implies AC \oind_C B$
        \item Extension: If $A \oind_C B$ and $\hat B \supset B$ then there is $A' \equiv_{BC} A$
        such that $A' \oind_C \hat B$
        \item Finite Character: If $A_0 \oind_C B$ for all finite $A_0 \subset A$ then $A \oind_C B$
        \item Local Character: For every $A$ there is cardinal $\kappa(A)$ such that for any $B$,
        there is a subset $C \subset B$ of cardinality $|C| < \kappa(A)$ such that $A \oind_C B$
        \item Anti-Reflexivity: $a \oind_B a$ implies $a \in\acl(B)$
    \end{itemize}
\end{defn}

\begin{defn}
	We say an independence relation $\oind$ has the \textit{intersection property} if whenever $C_1,C_2 \subseteq B$ such that
    $A \oind_{C_1} B$ and $A \oind_{C_2} B$ then
    \[
    A \oind_{\acl^\eq(C_1) \cap \acl^\eq(C_2)} B.
    \]
    An independence relation (which, here, is necessarily anti-reflexive) is called \emph{canonical} just in case it has the intersection property.
\end{defn}

An independence relation having the intersection property implies the independence relation is
in fact  \th-forking.
\begin{thm}[\cite{adler-thesis}, Theorem 3.3]\label{thm:intersection-is-thorn-forking}
    If $\oind$ is a non-trivial independence relation with the intersection property,
    then $\oind=\thind$.
\end{thm}
The converse of this theorem does not hold: there are (non-simple) theories where \th-forking does not have the intersection property \cite{adler-thesis}. On the other hand, it is shown in \cite{adler-thesis} that $\indep$ is canonical provided that the simple theory in question eliminates hyperimaginaries.  Whether $\thind$ is canonical in \emph{every} simple theory is still beyond our knowledge, but we do have the following generalization:

\begin{thm}\label{thm:dwip-intersection}
    Suppose $T$ is simple and has WEHI.
    Then $\indep$ has the intersection property.
\end{thm}
\begin{proof}
    Choose any sets $A$ and $B$.
    Suppose there are subsets $C_1,C_2 \subseteq B$ with
    $A \indep_{C_1} B$ and $A \indep_{C_2} B$.
    We need to show 
    $A \indep_{\acl(C_1) \cap \acl(C_2)} B$.

    Since $T$ is simple there is hyperimaginary $e = \Cb(A/B)$ and $e \in \bdd(B)$.
    Moreover, $e \in \bdd(C_1) \cap \bdd(C_2)$, by using properties of canonical bases.
    Using DWIP we get a set of imaginary elements $\hat e \subset \dcl(e) \cap \M^\eq$ such that
    $e \subset \bdd(\hat e)$ and for all sets $U$,
    $U \indep_{\hat e} e$.
    In particular, $A \indep_{\hat e} e$.
    Since $\hat e \subset \dcl(e) \subset \bdd(C_1)$
    and $\hat e$ is a tuple of real elements, $\hat e \subseteq \acl(C_1)$.
    The same reasoning also shows $\hat e \subseteq \acl(C_2)$.
    Thus $\hat e \subseteq \acl(C_1) \cap \acl(C_2)$.
    Moreover, transitivity of forking gives
    \[
    A \indep_{\hat e} e \quad \land \quad A \indep_e B \implies A \indep_{\hat e} eB.
    \]
    And so,
    $A \indep_{\acl(C_1) \cap \acl(C_2)} B$
    since $\hat e \in \acl(C_1) \cap \acl(C_2)$.
\end{proof}
\begin{cor}\label{cor:dwip-forking-is-thorn}
    If $T$ is simple and has WEHI, forking is equivalent to \th-forking and $\thind$ is canonical.
\end{cor}
\begin{proof}
    Combine Theorems~\ref{thm:dwip-intersection} and~\ref{thm:intersection-is-thorn-forking}.
    Since forking has the intersection property and is equivalent to \th-forking, \th-forking also
    has the intersection property.
\end{proof}

\subsection{When forking and \th-forking coincide}

The results in this section address the weak elimination of hyperimaginaries,
specifically we show the converse of Corollary~\ref{cor:dwip-forking-is-thorn}.

In this subsection, we demonstrate the converse of Corollary \ref{cor:dwip-forking-is-thorn}
\begin{thm}\label{thm:int-wehi}
Let $T$ be simple.
If forking and \th-forking are equivalent (on imaginaries)
and \th-forking is canonical, then $T$ has WEHI.
\end{thm}
The issue is that we only assume the equivalence of the two notions on {\em imaginary elements} -- it is not \emph{a priori} clear that that this equality should be enough to reduce hyperimaginaries to elements of $M^\eq$.

We will use $\acl^\eq$ and $\dcl^\eq$ to indicate when the closure operations are on $M^\eq$, and assume the unadorned versions
$\acl$ and $\dcl$ always include hyperimaginaries.

The theorem is done as a sequence of lemmas.
The first is easy, and is essentially identical to Lemma 3.6.4 in \cite{Wagner-SimpleTheories}.
\begin{lemma}\label{lem:wecb-wehi}
If $T$ weakly eliminates every hyperimaginary {\em arising as a canonical base}, then
$T$ weakly eliminates every hyperimaginary.
\end{lemma}
\begin{proof}
Consider an arbitrary hyperimaginary $a_E$ (where $a$ may be an infinite tuple here)
and $E$ is a type-definable equivalence relation.
Let $d = \Cb(a/a_E)$.
Then, $a \indep_d a_E$.
By hypothesis, there is a set of imaginary elements $C$ such that $\bdd(C) = \bdd(d)$,
so also $a \indep_C a_E$.
Since $a_E \in \dcl(a)$, we have $a_E \in \bdd(C)$, but also, by definition $d \in \dcl(a_E)$, 
so $C \subset \bdd(d) \subset \bdd(a_E)$.
Hence $\left.\bdd(C) = \bdd(a_E)\right.$,
and $a_E$ is weakly eliminated.
\end{proof}

A possible objection to the above proof is that when we talk of a canonical base
of a type, we are thinking of a type over a domain of imaginary elements and not a hyperimaginary
element as used in the previous proof.
This is not, in fact, an issue because of the following lemma.
\begin{lemma}\label{lem:wecb-wehcb}
If $T$ weakly eliminates canonical bases of types over imaginaries,
   then $T$ also weakly eliminates canonical bases of types over hyperimaginaries.
\end{lemma}
\begin{proof}
Let $p(x,e)$ be a type over hyperimaginary $e$.
Then $p$ is equivalent to a partial type $\pi(x,\hat e)$ over a possibly infinite imaginary tuple $\hat e$ that is
a member of the equivalence class named by $e$.
Let $p'\in S(\hat e)$ be a completion of $\pi$ that does not fork over $e$.
Then $p'$ is a non-forking extension of $p$, and so
%
$\bdd(\Cb(p')) = \bdd(\Cb(p))$.
Since $p'$ is defined over a set of imaginary elements, the hyperimaginary $d = \Cb(p')$ is weakly eliminated.
Hence the canonical base of $p$, being interbounded with $d$, is also weakly eliminated.
\end{proof}
In the presence of the intersection property, there is a smallest algebraically closed set over which a type does not fork.
We call this set the weak canonical base.
\begin{defn}
Let $p$ be a type over $B \subset M^\eq$.
The \textit{weak canonical base}, $\WCb(p)$ is the smallest algebraically closed subset (of imaginaries) $C \subset B$ such that
$p$ does not fork over $C$.
\end{defn}
The following properties for weak canonical bases parallel those for canonical bases.
\begin{prop}[\cite{adler-thesis}, Exercise 3.21]
    For a tuple $a$ and sets $B$,$C$ the following are equivalent.
\begin{enumerate}
\item $\bar a \indep_C B$
\item $\WCb(\bar a / BC) \subset \acl (C)$
\item $\WCb(\bar a / BC) = \WCb(\bar a/C)$
\end{enumerate}
\end{prop}


Our goal is to show that weak canonical bases, when they exist, behave enough like canonical bases so as to
(weakly) eliminate the latter. For our work, we first recall two known facts.

\begin{thm}[\cite{adler-thesis}, Theorem 3.20]
$\indep$ has the intersection property iff every type over an algebraically closed subset has a weak canonical base.
\end{thm}

\begin{prop}[\cite{Wagner-SimpleTheories}, Theorem 2.5.4]
    Let $\pi(x,a)$ be a partial type over $Aa$ that does not fork over $A$.
    If $(a_i : i < \omega)$ is a sequence in $\tp(a/A)$ indiscernible over $A$,
    then $\bigcup_i \pi(x,a_i)$ is consistent and does not fork over $A$.
\end{prop}

\begin{lemma}\label{lem:intprop-wecb}
If $\indep$ has the intersection property, then for each type $p$ over imaginary elements,
   the hyperimaginary $e = \Cb(p)$ is weakly eliminated.
\end{lemma}
\begin{proof}
Let $p(x)$ be a type over a set $B$ of imaginaries.
Let $D = \WCb(p)$, so $D \subset B$ is a set of imaginary elements and $p$ does not fork over $D$.
On the other hand, there is a canonical base $e = \Cb(p)$ which exists as a hyperimaginary element.
By definition, since $p$ does not fork over $D$, we have $e \in \bdd(D)$.
We wish to reverse this and show $D \subset \bdd(e)$.

Aiming for a contradiction, suppose $D \not\subset \bdd(e)$.
Then there is an $e$-indiscernible sequence $(D_i : i < \omega)$
of pairwise distinct realizations of $\tp(D/e)$ with $D_0 = D$.
Let $p'(x,D)$ be the restriction of $p$ to $D$.
Since $p$ does not fork over $e$, neither does $p'(x,D)$.
For $i < \omega$, $p'(x,D_i)$ is a conjugate over $e$ of $p'(x,D)$.
Since $p'(x,D)$ does not fork over $e$, the union $q(x) = \cup p'(x,D_i)$
is consistent and for each $i$, $q$ is a non-forking extension of $p'(x,D_i)$.
In particular this means $C = \WCb(q)$, the weak canonical base of $q$,
is contained in $D_i$ for each $i$.
In fact $C = D_i$ for each $i$ since $\WCb(p'(x,D_i)) = D_i$.
And so $D_i = C = D_j$ for every pair $i,j < \omega$.
But this contradicts the sequence consisting of distinct realizations of $\tp(D/e)$.
Hence $D \subset \bdd(e)$, and $e$ is weakly eliminated.
\end{proof}

Finally, we come to the proof of Theorem~\ref{thm:int-wehi}, stated at the beginning of the section.
\begin{proof}[Proof of Theorem~\ref{thm:int-wehi}]
Assume the hypothesis.
Then forking has the intersection property, and by Lemma~\ref{lem:intprop-wecb},
canonical bases of types over imaginary elements are weakly eliminated.
Then by Lemmas~\ref{lem:wecb-wehcb} and~\ref{lem:wecb-wehi}, all hyperimaginaries are weakly eliminated.
\end{proof}

A slightly more satisfying version of Theorem \ref{thm:int-wehi}, would following immediately from the statement of the following conjecture, with which we conclude this section.

\begin{conj}
In any simple theory, $\thind$ is canonical. 
\end{conj}

%
%
%

\section{Stable-forking}

\subsection{Symmetry of Stable Forking}
This short section presents a proof that stable forking is symmetric in all simple theories.
It builds on a foundation laid in \cite{Shami11}, where stable forking symmetry is proved for theories where Lascar strong type is equivalent to strong types.
We will recall the bare essentials of that paper, but only present a proof of the new result which we require.
This proof and all the other material in this subsection are due to Shami.

Here is a rough outline of the original proof.
First, one establishes that for stable formulas, generic satisfiability is unique for parameters that have the same Lascar strong type.
\begin{fact}\label{fact:lstp-gives-same-stable-set}\cite[Claim 6.5]{Shami11}
    Let $T$ be simple, and let $\phi(x,y) \in L$ be stable.
    Assume $a \indep_A b$ and $a' \indep_{A} b$
    and that $\lstp(a/A) = \lstp(a'/A)$.
    Then $\models\phi(a,b)$ if and only if $\models\phi(a',b)$.
\end{fact}
One uses this fact to identify an alternate definition of forking (more akin to forking in a stable theory) with the usual one for simple theories.

\begin{defn}
Let $\phi(x,y)$ be a formula, and let $A\subseteq B\subset\MM^\eq$. Then $p\in S_\phi(B)$ \emph{does not fork in the sense of LS} if, for some model $M$ containing $B$, some $p'\in S_\phi(M)$ extending $p$ is definable over $\acl(A)$.
\end{defn}

\begin{fact}\cite[Lemma 6.6]{Shami11}
    Assume $T$ is a simple theory with $\lstp = \stp$, and
    let $\phi(x,y) \in L$ be a stable formula. Let $a\in\MM^\eq$ and $A\subseteq B\subset\MM$. Then 
    $\tp_\phi(a/B)$ does not fork over $A$ in the sense of LS if and only if
    $\tp_\phi(a/B)$ does not fork over $A$.
\end{fact}
Provided $\lstp = \stp$, it is a short step from this to the symmetry of ``stable non-forking independence,'' from which the symmetry of stable forking is fairly obvious.
Throughout Shami's argument, the assumption that $\lstp = \stp$ appears only so that fact \ref{fact:lstp-gives-same-stable-set} may be used.
Thus, to extend his argument to arbitrary simple theories, it suffices to prove an analog of fact \ref{fact:lstp-gives-same-stable-set} for strong types in place of Lascar strong types.
Strangely, the proof of this extension actually uses its precursor.

\begin{lemma}\label{lemma:cotypedef-eq-rel}
    Let $E(x,y)$ be a bounded, co-type definable equivalence relation.
    Then $E$ is a definable, finite equivalence relation.
\end{lemma}
\begin{proof}
    Suppose $E$ is as in the hypothesis.
    Let $r(x,y)$ be a partial type defining $\neg E$.
    If $E$ had an infinite number of equivalence classes, then
    we could build an indiscernible sequence where each 2-type realizes $r$.
    The sequence would then imply the number of $E$ classes is unbounded, contradicting the hypothesis.
    Thus $E$ has only a finite number of classes.
    Let $a_1,\ldots,a_n$ for some $n < \omega$ consist of a single representative from each class.
    Then the type $\bigwedge_{i\leq n} r(x,a_i)$ is inconsistent, so there is a formula $\psi(x,y) \in r$ such that $\bigwedge_{i \leq n} \psi(x,a_i)$ is inconsistent by compactness.
    Observe that $\neg\exists x(\bigwedge_{i \leq n} \psi(x,a_i))$
    is equivalent to $\forall x(\bigvee_{i \leq n} \neg\psi(x,a_i))$,
    so for every $x$ there is some $j$ such that $\neg\psi(x,a_j)$.
    Since $\psi \in r$, $\neg \psi(x,y) \models E(x,y)$.
    Let
    \[
    \theta(x,y) \equiv \bigwedge_{i\leq n} \left(\neg\psi(x,a_i) \bic \neg\psi(y,a_i)\right).
    \]
    We will show $\theta$ defines $E$.
    First suppose $\theta(b,c)$.
    For some $j$ we have $\neg\psi(b,a_j)$, and $\theta$ entails $\neg\psi(c,a_j)$.
    Thus $E(b,a_j)$ and $E(c,a_j)$ so by transitivity $E(b,c)$.
    
    Conversely, suppose $E(b,c)$.
    If $\neg\theta(b,c)$ then there is some index $i$ such that (without loss of generality) $\neg\psi(b,a_i) \land \psi(c,a_i)$ holds.
    There is another index $j$ such that $\neg\psi(c,a_j)$.
    Then we have $E(b,a_i)$, $E(b,c)$ and $E(c,a_j)$.
    By transitivity $E(a_i,a_j)$.
    But we chose $a_i$ and $a_j$ to be in different classes.
    The contradiction proves the claim and the lemma.
\end{proof}

\begin{lemma}\label{lemma:stp-gives-same-stable-set}
    Let $T$ be simple.
    Let $\phi(x,y) \in L$ be stable.
    Assume $a \indep_{A} b$ and $a' \indep_{A} b$
    and that $\stp(a/A) = \stp(a'/A)$.
    Then $\models\phi(a,b)$ if and only if $\models\phi(a',b)$.
\end{lemma}
\begin{proof}
     Given a complete type $q(x)$ define the equivalence relation $E_q$ by
    \[
    E^\phi_q(a,a') \iff \text{for every $b \models q$ with $b \indep aa'$, $\models \phi(a,b) \bic \phi(a',b)$}.
    \]
    The complement of $E_q^\phi$ is defined by a partial type over $A$.
    By fact \ref{fact:lstp-gives-same-stable-set},
    if $\lstp(a) = \lstp(a')$, then $E_q^\phi(a,a')$ holds; this implies $E^\phi_q$ is refined by the equality of Lascar strong type,
    so it is bounded.
    A bounded, co-type-definable equivalence relation is a finite, definable equivalence relation by lemma \ref{lemma:cotypedef-eq-rel}, so $E^\phi_q$ is an $A$-definable finite equivalence relation.
   
    Now, suppose $\stp(a/A) = \stp(a'/A)$, $b \indep_{A} a$ and $b \indep_{A} a'$.
    Put $q = tp(b/A)$.
    By definition, $\tp(a/\acl^\eq(A)) = \tp(a/\acl^\eq(A))$, so $a,a'$ must be in the same $E_q^\phi$-class.
    Let $a''$ realize $\stp(a/A)$ such that $a'' \indep_{A} aa'b$.
    By transitivity, $aa'' \indep_{A} b$ and $a'a'' \indep_{A} b$.
    Then, since $E_q^\phi(a,a'')$ and $E_q^\phi(a',a'')$ we have
    $\phi(a,b) \bic \phi(a'',b) \bic \phi(a',b)$, as desired.
\end{proof}

The remainder of Shami's argument (Lemma 6.6 to the end of section 6 of \cite{Shami11}) goes through essentially unchanged except for substituting our lemma \ref{lemma:stp-gives-same-stable-set} for fact \ref{fact:lstp-gives-same-stable-set}.
\begin{thm}\label{stableforkingsymmetry}
    Let $T$ be simple.
    If $a \indep_C b$ and there is a stable formula $\psi(x;y)$ such that $\psi(x;b)$ forks over $C$, then there is a stable formula $\varphi(x;y)$ such that $\varphi(x;a)$ forks over $C$.
\end{thm}

\subsection{Stable Forking with  Elements of rank $\leq 2$}

\subsubsection{Quasi-designs, and stability in there}

The eventual argument will require there to be an indiscernible sequence in $\tp(b/aC)$
such that $a$ is not in the algebraic closure of the sequence with $C$.
Such sequences may not always exist.
To this end, we will find some sufficient conditions for which they will exist.
The property we are concerned about is
\begin{center}
\begin{tabular}{lp{4in}}
    ($\maltese$) & $a \notin \acl(bC)$, $b \notin\acl(aC)$, and for every non-constant $aC$-indiscernible sequence $I$ in
        $\tp(b/aC)$, $a \in \acl(IC)$.
\end{tabular}
\end{center}
We will speak of elements $(a,b,C)$ satisifying property $\maltese$ in the natural way, and we will denote this condition by $\maltese(a,b,C)$.

\begin{prop}\label{prop:maltese}
\mbox{}
\begin{enumerate}
\item In the definition of $\maltese$ it is equivalent to require for every infinite set $D$ of distinct realizations
    of $\tp(b/aC)$ we have $a \in\acl(DC)$.

\item The property $\maltese$ is co-type-definable.
That is, there is a partial type $\Omega(x,y,z)$ such that for all $a,b,C$ (with $|a|=|x|,|b|=|y|,|C|=|z|$)
we have $\models \Omega(a,b,C)$ if and only if $(a,b,C)$ do not have $\maltese$.

\item If $\maltese(a,b,C)$ holds then there are formulas $\delta(x,z)$, $\zeta(y;xz)$, $\theta(x;zy_0\dots y_n)$ such that
    for some finite $c \subset C$
    \begin{enumerate}
        \item $\delta(x,z) \in \tp(ac)$
        \item $\zeta(y;ac) \in \tp(b/ac)$
        \item $\theta(x;z\bar y)$ is algebraic in $x$
        \item The following entailment holds
            \[
\delta(x,z) \vdash \forall y_0\dots y_n \left[
\left(\bigwedge_{i<j<n} y_i \neq y_j \land \bigwedge_i \zeta(y_i; xz)\right)
\cond \theta(x;zy_0\dots y_n)\right]
            \]
        \end{enumerate}
\end{enumerate}
\end{prop}
\begin{proof}
(1) If $a$ is algebraic in every infinite set of realizations of $\tp(b/aC)$ and $C$ then certainly $a$ is algebraic
over any indiscernible sequence in $\tp(b/aC)$ and $C$.
Conversely, if there is some infinite set $B$ of realizations of $\tp(b/aC)$ for which $a$ is not algebraic over $BC$,
then we can, using compactness, get an arbitrarily large set $B' = (b_\alpha : \alpha < \lambda)$ of
realizations of $\tp(b/aC)$ such that $a \notin\acl(B'C)$.
From $B'$ we may extract a sequence $I$ which is indiscernible over $aC$.
But then $a \notin\acl(IC)$.

(2) Since the statement $x \notin\acl(z)$ is a partial type in $x$ and $z$, the desired type says
that there are infinitely many elements realizing $\tp(y/xz)$ over which $x$ is not algebraic.
\[
\exists y_0y_1\dots \left[x \notin\acl(z y_0y_1\dots)
    \land \bigwedge_i \tp(y_i/xz) = \tp(y/xz)
    \land \bigwedge_{i\neq j} y_i \neq y_j \right]
\]

(3) If $\maltese(a,b,C)$ holds then the partial type in (2) is inconsistent.
Hence there is some finite $c \subset C$, a formula $\zeta(y,ac) \in \tp(b/aC)$,
and an algebraic formula $\theta(x;zy_0 y_1\dots y_n)$, for some $n < \omega$,
such that
\[
\mathrm{ediag}(ac) \vdash \forall y_0\dots y_n \left[
\left(\bigwedge_{i<j<n} y_i \neq y_j \land \bigwedge_i \zeta(y_i; ac)\right)
\cond \theta(a;cy_0\dots y_n)\right]
\]
The right hand side is a formula, hence there is a formula $\delta(x,z) \in \tp(ac)$
which implies it, by compactness.
\end{proof}

A direct application of part (3) above shows that over the empty set,
there is always stable forking.
\begin{prop}\label{prop:maltese-over-empty}
If $\maltese(a,b,\emptyset)$, then there is a stable formula $\psi(x;b) \in \tp(a/b)$ which forks over $\emptyset$.
\end{prop}
\begin{proof}
Suppose $\maltese(a,b,\emptyset)$ holds.
From proposition \ref{prop:maltese} there are formulas $\delta(x)$,$\zeta(y,x)$, and $\theta(x;y_1\dots y_n)$ for some $n < \omega$.
Since $\maltese$ holds over the empty set, the ``finite tuple'' in question is null; hence, the formulas reduce to
the following entailment.
\begin{equation}\label{eq:delta-x}
\delta(x) \vdash \forall y_0\dots y_n \left[
\left(\bigwedge_{i<j<n} y_i \neq y_j \land \bigwedge_i \zeta(y_i; a)\right)
\cond \theta(a;y_0\dots y_n)\right]
\end{equation}
Consider the formula $\psi(x;b) \equiv \delta(x) \land \zeta(b;x) \in \tp(a/b)$.
We will show $\psi$ is a stable formula and $\psi(x;b)$ forks over $\emptyset$.

\underline{Stable:} Suppose $\psi$ is not stable, and let $I = (d_i e_i)_{i \in \ZZ}$ be an
indiscernible sequence witnessing this by $\models \psi(d_i; e_k)$ iff $i \leq k$.
(Note that this sequence is necessarily non-constant.)
Then $\delta(d_i)$ holds for all $i$.
Consider the element $d_{-1}$.
Since $\psi(d_{-1};e_k)$ holds for $1 \leq k \leq n$,
so does $\zeta(e_k; d_{-1})$.
Hence, from the entailment in equation (\ref{eq:delta-x}),
$\theta(d_{-1}; e_1 e_2 \dots e_n)$ is true, which witnesses
$d_{-1} \in \acl(e_1 e_2 \dots e_n)$.
This contradicts that $I$ is an indiscernible sequence.
Hence $\psi$ is stable.

\underline{Forking:} Suppose $\psi(x,b)$ does not fork over $\emptyset$.
Let $r(x,b)$ be a non-forking extension of the partial type $\{\psi(x,b)\}$,
and let $c \models r(x,b)$ realize it, so $c \indep b$.

Let $I = (b_i : i<\kappa)$ be an indiscernible sequence
such that $b_0=b$ and $\kappa$ is sufficiently large for the remainder.
Note that $r'(x,I) = \left\{r(x,b_i):i<\kappa\right\}$ is consistent, and we may
choose $c'\models r'(x,I)$ such that $c'\equiv_b c$ and $c'\indep I$.
As, by assumption, $\kappa$ is large enough, there is an infinite subset $I_0\subset
I$ which is indiscernible over $c'$.
Without loss of generality, we assume that $I_0 = (b_i : i<\omega)$.
Now, note that $\psi(c',b_i)$
for each $i<\omega$, so in particular, $\models \zeta(c',b_0,...,b_N)$.
It follows that $c'\in\acl(b_0...b_N)$, and since $c'\indep b_0...b_N$, we have
$c'\in\acl(\emptyset)$.
This implies $c\in\acl(\emptyset)$.

Now, if \emph{every} non-forking extension $\tp(c^*/b)$ of  $\{\psi(x,b)\}$
induces $c^*\in \acl(\emptyset)$,
then $\models\psi(x,b) \implies x\in\acl(b)$.
Since $\psi(x,b) \in\tp(ab)$, this contradicts the assumption that $a\notin\acl(b)$.
Thus, $\psi(x,b)$ forks over $\emptyset$.
\end{proof}

A configuration related to $\maltese$ is the \textit{quasidesign}.
A \textit{quasidesign} is a partial type $r(x,y)$ such that
for any $a,b \models r$ we have
$a \notin\acl(b)$, $b\notin\acl(a)$, and
for all distinct $b_1,b_2$ the set
\[
\{ a' : \text{$\models r(a',b_1)$ and $\models (a',b_2)$}\}
\]
is finite.
A \textit{pseudoplane} is a quasidesign where also for all distinct
$a_1,a_2$ the set $\{ b' : \text{$\models r(a_1,b')$ and $\models (a_2,b')$}\}$
is also finite.
A theory \textit{omits quasidesigns} (pseudoplanes) if there is no partial type
which is a quasidesign (pseudoplane).
Quasidesigns, or rather, the lack thereof, are equivalent to 1-basedness.
\begin{thm}[see   \cite{pillay-geom-book}]
The following are equivalent.
\begin{enumerate}
\item $T$ is 1-based
\item $T$ admits no quasidesigns
\item $T$ admits no pseudoplanes
\end{enumerate}
\end{thm}

The salient property of quasidesigns which relates them to $\maltese$ is
the following
\begin{prop}
Suppose $r(x,y)$ is a quasidesign over $C$,
        $b$ is an element and $(a_i : i <\omega)$ is a set of elements such that
        $r(a_i,b)$ holds for all $i$.
        Then $b \in \acl(C(a_i )_{ i < \omega})$.
\end{prop}
\begin{prop}
If $r(x,y)$ is a quasidesign over $C$, and $a,b$ are elements such that
$r(a,b)$ holds, then $\maltese(b,a,C)$ holds.
\end{prop}
\begin{proof}
Let $r$ be a quasidesign and $\models r(a,b)$.
By definition $a \notin\acl(Cb)$ and $b \notin\acl(Ca)$.
To show $\maltese(b,a,C)$ we only need to show that for any
indiscernible sequence $I$ in $\tp(a/Cb)$ we have $b \in\acl(CI)$.
Yet, for any such sequence $I = (a_i : i < \omega)$ we have
$\models r(a_i,b)$ for all $i$.
Hence $b \in \acl(CI)$.
\end{proof}

In the next subsection, we will use these facts to demonstrate that forking \emph{over $\emptyset$} between an element/tuple of SU-rank 2 and one of finite SU-rank is always witnessed by a stable formula. The obstruction to removing the restriction to $\emptyset$ in the base is Proposition \label{prop:maltese-over-empty}, which we have not been able to generalize.

\subsubsection{Stable Forking with  Elements of rank $\leq 2$}

At last, we come to the proof that forking between a finite-rank element and one of SU-rank $\leq 2$ is always witness by a stable formula. The main theorem of this section -- theorem \ref{thm:rank2stable} -- is a good illustration of the use of DWIP, which is why we include it here, but this usage requires that bit of additional work carried out in the previous subsection on quasidesigns.
That main theorem treats elements of SU-rank exactly 2, so to start with, we deal with ranks $<2$.

\begin{prop}\label{thm:rank-1}
    Let $C\subset\MM$, and let $a,b\in\MM$ be finite tuples.
    If $SU(b/C) = 1$ and $a \nindep_{C} b$ then there are stable forking formulas in both $tp(a/bC)$ and $tp(b/aC)$.
\end{prop}
\begin{proof}
    From $a \nindep_{C} b$, we have $SU(b/aC) < SU(b/C) = 1$, so $b\in\acl(aC)\setminus \acl(C)$.
Let $\theta(y;ac) \in tp(b/aC)$ be an algebraic formula.
    Then $\theta(y;ac)$ forks over $C$ because $b \notin \acl(C)$, and it is stable because it is algebraic.
    The stable forking formula inside $\tp(b/aC)$ then follows from theorem \ref{stableforkingsymmetry}.
\end{proof}

We next show stable forking can be passed ``upward'' through algebraic closure.
It is similar to the result of Kim \cite{KimInThere} that if $E(y;z)$ is a finite equivalence relation and $\varphi(x;y)$ is any formula then $\exists y[\varphi(x;y)\land E(y;z)]$ is stable.
For this one claim we do not need to assume $T$ is simple.

\begin{lemma}
    Let $T$ be an arbitrary theory.
    Suppose $\zeta(x;y)$ is a stable formula
    and $\theta(y;zw)$ is algebraic in $y$.
    Let $\psi(x;zw)$ be the formula $\exists y[\zeta(x;y) \land \theta(y;zw)]$.
    Then,
    \begin{enumerate}
        \item $\psi(x;zw)$ is stable;
        \item if there are elements $a,b,c,d$ and a set $D$ containing $d$ such that 
        $a \models \zeta(x;c)$, $\zeta(x;c)$ forks over $D$, and $\theta(x;bd)$ isolates $\tp(c/bD)$ (so $c \in \acl(bD)$)
        then $\psi(x;b)$ forks over $D$.
    \end{enumerate}
\end{lemma}
\begin{proof}[Proof of (1)]
    Towards a contradiction, suppose $\psi(x;zw)$ is unstable.
    Let $(a_ib_ic_i : i < \omega + \omega)$ be an indiscernible sequence witnessing the order property -- i.e. $\models \psi(a_i;b_jc_j)$ iff $i \leq j$.
    For $i < \omega$, let $d_i$ be the element
    witnessed by the existential in $\psi(a_i; b_\omega c_\omega)$.
    Since $\theta(y;b_\omega c_\omega)$ is algebraic, there are only finitely many possible $d_i$,
    so by the pigeonhole principle at least one, $d'$, is repeated infinitely often.
    As $\zeta$ is stable and $(a_i)_{i<\omega+\omega}$ is indiscernible, the set $\left\{i<\omega{+}\omega: \,\,\models\zeta(a_i;d')\right\}$ is either finite or cofinite,     and by our choice of $d'$, it must be cofinite.
    Consequently, there are indices $k > \omega$ such that $\models \zeta(a_k;d')$, and this entails $\models\psi(a_k; b_\omega c_\omega)$,
    a contradiction .
\end{proof}

\begin{proof}[Proof of (2)]
    Let $(c_i)_{i<n}$ enumerate all $n$ elements satisfying $\theta(y;bd)$.
    We may assume all the $c_i$ have the same type over $bD$ as $c$
    since $\theta$ isolates $\tp(c/bD)$.
    Since $\zeta(x;c)$ forks over $D$, so does $\bigvee_{i<n} \zeta(x;c_i)$.
    As $\psi(x;bd) \vdash \bigvee_{i<n} \zeta(x;c_i)$,  $\psi(x;bd)$ forks.
\end{proof}

We restate the previous lemma in a more useful way.
This corollary is what we mean by stable forking passing ``upward'' in algebraic closure.
The stable forking of $a$ with $d$ is passed to $a$ and $b$.
\begin{cor}\label{thm:stable-upward-in-acl}
    Let $T$ be an arbitrary theory.
    Suppose $a$ and $b$ are tuples where $tp(a/Cb)$ forks.
    Moreover, suppose there is an element $d \in acl(Cb)$ such that $tp(a/Cd)$ forks
    via a stable formula. Then $tp(a/Cb)$ also contains
    a stable forking formula.
\end{cor}

\begin{thm}\label{thm:rank2stable}
    Let $T$ be a simple theory with DWIP, suppose $SU(a) = 2$ and $SU(b)<\omega$. 
    Then, if $a \nindep b$ then there is a stable formula in $tp(a/b)$ which forks over $\emptyset$.
\end{thm}
\begin{proof}
    Suppose $a \nindep b$.
    The proof is by induction on the $SU$-rank of $b$ over $\emptyset$.
    If $SU(b) = 1$ then proposition \ref{thm:rank-1} produces a stable forking formula in $tp(a/b)$.
    Now suppose $SU(b) = r<\omega$ and the proposition holds for all elements
    of smaller rank.
    
    Now, either $\maltese(a,b,\emptyset)$ holds or it doesn't. If it does hold, then we recover a stable forking formula in $\tp(a/b)$ using Proposition \ref{prop:maltese-over-empty}, so we may assume that $\maltese(a,b,\emptyset)$ does not hold
    In particular, we choose a sequence $I = (b_ic)_{i \in\omega}$ containing $b$
    which is indiscernible over $a$ and for which $a \not\in \acl(I)$.
    The dependence $a \nindep b$ implies $a \nindep I$ which yields the rank inequalities:
    $$1 \leq SU(a/I) \leq SU(a/b) < SU(a) = 2.$$
    Hence, $SU(a/I) = SU(a/b)$, and so $a \indep_{b} I$.
    Put $e = \Cb(a/I)$.
    Since $a$ and $I$ are independent over $b$, $e \in \bdd(b)$.
    In fact, we could have done the above rank argument with any $b' \in I$ to give $e \in \bdd(b')$.
    However, $b \notin \acl(e)$, since $b \in \acl(e)$ would imply $b \in \acl(b'C)$ for any $b' \in I$, which is impossible since $b$ and $b'$ are both in the same indiscernible sequence $I$.
    
    Hence $SU(b/e) > 0$ and $SU(e/b) = 0$.
    The Lascar inequalities then show $SU(e) < SU(b)$:
    \[
    SU(e) < SU(b/e) + SU(e) \leq SU(eb) \leq SU(e/b) \oplus SU(b) = SU(b)
    \]
    Both $a \nindep I$ and $a \indep_{e} I$ imply $a \nindep e$ by transitivity.
    Applying DWIP yields a finite, real tuple $d \in \acl(e)$ with $a \nindep d$.
    The induction hypothesis gives a stable forking formula inside $tp(a/d)$
    since $SU(d) \leq SU(e) < SU(b)$.
    Then because $d \in \acl(b)$, Corollary \ref{thm:stable-upward-in-acl} gives a stable forking formula in $tp(a/b)$.
\end{proof}

We note that the statement  of Theorem \ref{thm:rank2stable} makes no explicit appeal to 1-basedness, but of course, the possible lack of 1-basedness that necessitated the discussion fo quasi-designs and property $\maltese$. We were only able to handle $\maltese$ with $\emptyset$ in the base because the definition of stable-forking requires that the stable formula is stable as a formula without parameters -- \emph{before} loading a tuple of elements from the base $C$. If we relax the definition of stable forking, we \emph{can} have a ``stronger'' theorem:

\begin{defn}
A simple theory $T$ has \emph{relatively stable forking} if whenever $q(x)$ is a complete type over a model $M$, $C \subset M$ and $q$ forks over $A$ then there is is a formula,
$\psi(x;b,c) \in q$ such that:
\begin{itemize}
\item If $c'\equiv c$, then $\psi(x;y,c')$ does not have order property \emph{with $c'$ fixed}.
\item $\psi(x;b,c)$ forks over $C$.
\end{itemize}
\end{defn}

\begin{thm}\label{thm:rank2stable-relative}
    Let $T$ be a simple theory with DWIP, suppose $SU(a/C) = 2$ and $SU(b/C)<\omega$. 
    Then, if $a \nindep_C b$, then there is a formula $\phi(x;b,c)\in tp(a/bC)$ which forks over $C$ and such that $\phi(x;y,c)$ does not have the order property with $c$ fixed.
\end{thm}

The other way -- and probably the more natural way -- to deal with property $\maltese$ is to just assume 1-basedness. This is the approach taken in the next subsection.

\subsection{A generalization of the argument to 1-based theories of finite SU-rank} 
\label{sec:broader-application}

In this section, we will show that the proof for theorem \ref{thm:rank2stable} has much broader applicability than it would initially seem. 
More precisely, in that argument, we used an indiscernible sequence $I = (b_i : i < \omega)$ with the following properties:
\begin{enumerate}
    \item $I$ is indiscernible over $aC$
    \item $a \indep_{b_i C} I$ for all $i < \omega$
\end{enumerate}
There we relied on that fact of $SU(a/C) = 2$ to make such a sequence, but as we noted earlier, this was akin to saying that ``$a$ is 1-based with respect to $I$.'' Here, we will see that this statement is precise; that is, assuming $T$ is 1-based is sufficient to extend the argument to all elements of finite SU-rank.
\begin{lemma}\label{lem:1based-sequence}
    Suppose $T$ is a simple, 1-based theory.
    If $a$,$b$ are elements such that $a \nindep_{C} b$ for some set $C$
    then there is sequence $I = (b_i : i < \omega)$ indiscernible over $aC$ such that $b_0 = b$ and
    satisfying $a \indep_{b_i C} I$ for all $i < \omega$.
\end{lemma}
\begin{proof}
    Let $a$,$b$,$C$ be as in the statement of the lemma.
    We will consider ``negative'' ordinals for indexing purposes, and
    let $J = (b_i : -|T|^+ \leq i \leq |T|^+)$ be a Morley sequence in $\tp(b/aC)$.
    By using an automorphism, we may assume $b_0 = b$.
    Let $L = (b_i : 1 \leq i \leq |T|^+ )$ and $K = (b_i : -|T|^+ \leq i \leq -1)$.
    For the moment we will focus on $L$.
    By 1-basedness $L$ is Morley over $b_0 C$.
    For some subset $D \subset L$ such that $|D| \leq |T|$ we have
    $a \indep_{b_0 C D} L$.
    Let $L^* = L \setminus D$.
    Since $L^*$ is Morley over $b_0 C$, $D \indep_{b_0 C} L^*$.
    Apply transitivity to get $a \indep_{b_0 C} L^*$.
    In the same way we find a subset $K^* \subset K$ such that $a \indep_{b_0 C} K^*$.
    There is a subset $D \subset K^*$ such that $a \indep_{b_0 C L^* D} K^*$,
    by replacing $K^*$ with $K^* \setminus D$ we may assume $a \indep_{b_0 C} K^* b_0 L^*$
    (note that in either case we have $|K^*| = |T|^+$).
    Let $I' = K^{*\frown} b_0^\frown L^*$.
    $I'$ is indiscernible over $aC$, which implies for each $b_i \in I$,
    $a \indep_{b_i C} I'$.
    Taking $I = (b_i : 0 \leq i < \omega) \subset I'$ then works for the conclusion of the lemma.
\end{proof}

Now we have the theorem.
\begin{thm}
    Let $T$ be a 1-based supersimple theory.
    Let $C$ be a set.
    Suppose $a$ is an imaginary element and $b$ is an arbitrary imaginary element with $SU(b/C) < \omega$.
    If $a \nindep_{C} b$ then there is a stable formula in $tp(a/bC)$ which forks over $C$.
\end{thm}
\begin{proof}
    This proof closely follows the proof of theorem \ref{thm:rank2stable}.
    Suppose $a \nindep_{C} b$.
    We proceed by induction on $SU(b/C)$.
    If $SU(b/C) = 1$ then proposition \ref{thm:rank-1} produces a stable forking formula in $tp(a/bC)$.
    Now suppose $SU(b/C) = \alpha$ and the proposition holds for all elements
    having smaller rank (over $C$).
    Choose a forking formula $\varphi(x;bc) \in \tp(a/bC)$ (for some $c \in C$), and a sequence $I = (b_ic)_{i \in\omega}$ containing $b$
    which is indiscernible over $aC$ and for which $a \not\in \acl(IC)$ and $a \indep_{bC} I$; such a sequence exists by lemma \ref{lem:1based-sequence}.
    Let $e = \Cb(a/IC)$.
    $e \in \bdd(bC)$ since $a \indep_{bC} IC$.
    In fact, the above rank argument could be repeated with any $b' \in I$ to give $e \in \bdd(b'C)$.
    However, $b \notin \acl(e)$, since $b \in \acl(e)$ would imply $b \in \acl(b'C)$ for any $b' \in I$, which is impossible since $b$ and $b'$ are both in the same indiscernible sequence $I$.
    Hence $SU(b/e) > 0$ and $SU(e/b) = 0$.
    The Lascar inequalities then show $SU(e) < SU(b)$.
    $a \nindep_{C} I$, and $a \indep_{e} I$ imply $a \nindep_{C} e$.
    DWIP produces a finite, real tuple $d \subset e$ with $a \nindep_{C} d$.
    Note that $SU(d) \leq SU(e)$ and $d \in \acl(bC)$.
    Since $SU(d C) < SU(bC)$, apply the induction hypothesis to get a stable forking formula inside $tp(a/\hat eC)$.
    Then corollary \ref{thm:stable-upward-in-acl} gives a stable forking formula in $tp(a/bC)$.    
\end{proof}

\bibliographystyle{plain}
\bibliography{myref}

\end{document}